\documentclass[11pt]{article}%
\usepackage{amsmath}
\usepackage{amsfonts}
\usepackage{amssymb}
\usepackage{graphicx}%
\newtheorem{theorem}{Theorem}

\newtheorem{lemma}[theorem]{Lemma}

\newtheorem{proposition}[theorem]{Proposition}

\newenvironment{proof}[1][Proof]{\noindent\textbf{#1.} }{\ \rule{0.5em}{0.5em}}
\begin{document}

\begin{center}
{\Large On Periodic and Chaotic Orbits in a Rational Planar System}

\medskip

\centerline{N. LAZARYAN and H. SEDAGHAT \footnote{Department of Mathematics, Virginia Commonwealth University Richmond, Virginia, 23284-2014, USA; Email: h.sedagha@vcu.edu}}
\end{center}

\medskip

\begin{abstract}
By folding an autonomous system of rational equations in the plane to a scalar
difference equation, we show that the rational system has coexisting periodic
orbits of all possible periods as well as stable aperiodic orbits for certain
parameter ranges.

\end{abstract}

\bigskip

The orbits of autonomous rational systems may exhibit behavior such as limit
cycles and chaos that are not seen in linear systems. It is difficult to prove
the existence of such solutions for nonlinear planar systems. In particular,
the existing literature seems to lack sufficient demonstrations of the
occurrence of complex behavior in rational systems. In this paper, we
investigate the semilinear rational system
\begin{subequations}
\label{1}%
\begin{align}
x_{n+1}  &  =ax_{n}+by_{n}+c\label{1a}\\
y_{n+1}  &  =\frac{a^{\prime}x_{n}+b^{\prime}y_{n}+c^{\prime}}{a^{\prime
\prime}x_{n}+b^{\prime\prime}y_{n}+c^{\prime\prime}} \label{1b}%
\end{align}

\noindent where all parameters are real numbers. To avoid reductions to linear
systems or to triangular systems, we assume that%

\end{subequations}
\begin{equation}
b\not =0,\quad|a^{\prime}|+|a^{\prime\prime}|,|a^{\prime\prime}|+|b^{\prime
\prime}|,|a^{\prime}|+|b^{\prime}|+|c^{\prime}|>0. \label{nt}%
\end{equation}

System (\ref{1}) in the case of non-negative parameters encompasses the
following 136 system types in \cite{CKLM} (and an equal number of mirror
systems obtained by switching x and y)
\begin{equation}
(7,l),\ (22,l),\ (25,l),\ (40,l)\quad\text{where }3\leq l\leq49,\ l\not =%
4,5,7,9,10,13,19,20,22,25,28,40\label{sty}%
\end{equation}

In this paper we study some of the global properties of (\ref{1}) through the
general process of folding. This concept appears, though not by this name, in
diverse areas from control theory to the study of chaos in differential
systems; see \cite{SF} for details and references and also for a general 
definition of folding as an algorithmic process. When folded to a
second-order rational difference equation, certain configurations of
parameters appear that are not readily apparent through other known methods.
We use these configurations to simplify the system and obtain sufficient
conditions for the occurrence of cycles and chaos.

\section{The main result}

Following \cite{SF}, we solve (\ref{1a}) for $y_{n}$ to obtain%
\begin{equation}
y_{n}=\frac{1}{b}(x_{n+1}-ax_{n}-c) \label{pyn}%
\end{equation}

Now%
\[
x_{n+2}=ax_{n+1}+by_{n+1}+c=c+ax_{n+1}+\frac{ba^{\prime}x_{n}+bb^{\prime}%
y_{n}+bc^{\prime}}{a^{\prime\prime}x_{n}+b^{\prime\prime}y_{n}+c^{\prime
\prime}}%
\]

Using (\ref{1b}) and (\ref{pyn}) to eliminate $y_{n}$ yields%
\[
x_{n+2}=c+ax_{n+1}+\frac{ba^{\prime}x_{n}+b^{\prime}(x_{n+1}-ax_{n}%
-c)+bc^{\prime}}{a^{\prime\prime}x_{n}+(b^{\prime\prime}/b)(x_{n+1}%
-ax_{n}-c)+c^{\prime\prime}}%
\]

Combining terms and simplifying we obtain the rational, second-order equation%
\begin{equation}
x_{n+2}=\frac{ab^{\prime\prime}x_{n+1}^{2}+aD_{ab}^{\prime\prime}x_{n+1}%
x_{n}+(aD_{cb}^{\prime\prime}+bb^{\prime}+cb^{\prime\prime})x_{n+1}%
+(bD_{ab}^{\prime}+cD_{ab}^{\prime\prime})x_{n}+bD_{cb}^{\prime}%
+cD_{cb}^{\prime\prime}}{b^{\prime\prime}x_{n+1}+D_{ab}^{\prime\prime}%
x_{n}+D_{cb}^{\prime\prime}} \label{xo2}%
\end{equation}
where%
\begin{equation}
D_{ab}^{\prime}=a^{\prime}b-ab^{\prime},\quad D_{ab}^{\prime\prime}%
=a^{\prime\prime}b-ab^{\prime\prime},\quad D_{cb}^{\prime}=bc^{\prime
}-b^{\prime}c,\quad D_{cb}^{\prime\prime}=bc^{\prime\prime}-b^{\prime\prime}c
\label{dets}%
\end{equation}

We refer to the pair of equations (\ref{pyn}) and (\ref{xo2}) as a
\textit{folding} of (\ref{1}). Note that (\ref{pyn}) is a \textit{passive}
equation in the sense that it yields $y_{n}$ without further iterations once a
solution $\{x_{n}\}$ of (\ref{xo2}) is known. In this sense, we may think of
(\ref{xo2}) as a reduction of (\ref{1}) to a scalar difference equation. If
$(x_{0},y_{0})$ is an initial point of an orbit of (\ref{1}) then the
corresponding solution of (\ref{xo2}) with initial values%
\begin{equation}
x_{0}\ \text{and }x_{1}=ax_{0}+by_{0}+c \label{iv}%
\end{equation}
yields the x-component of the orbit $\{(x_{n},y_{n})\}$ and the y-component is
given (passively) by (\ref{pyn}).

Equation (\ref{xo2}) is a rational equation of the type studied in
\cite{DKMOS}. The existence of solutions for (\ref{xo2}) is a nontrivial
issue. We consider the special case of (\ref{xo2}) where in addition to
conditions (\ref{nt}) the equalities $D_{ab}^{\prime},D_{ab}^{\prime\prime
},D_{cb}^{\prime\prime}=0$ hold, i.e.,%
\begin{equation}
a^{\prime}b=ab^{\prime},\ a^{\prime\prime}b=ab^{\prime\prime},\ b^{\prime
\prime}c=bc^{\prime\prime}. \label{d0}%
\end{equation}

In this case, (\ref{xo2}) reduces to the first-order recursion%
\[
x_{n+2}=\frac{ab^{\prime\prime}x_{n+1}^{2}+(bb^{\prime}+cb^{\prime\prime
})x_{n+1}+bD_{cb}^{\prime}}{b^{\prime\prime}x_{n+1}}%
\]
which by substituting $r_{n}=x_{n+1}$ may be written as%
\begin{align}
r_{n+1}  &  =ar_{n}+q+\frac{s}{r_{n}},\label{o1}\\
\text{where }q  &  =c+\frac{bb^{\prime}}{b^{\prime\prime}},\ s=\frac
{bD_{cb}^{\prime}}{b^{\prime\prime}},\ r_{0}=x_{1}\nonumber
\end{align}

A comprehensive study of Equation (\ref{o1}) appears in \cite{DKMOS}. In
particular, the following is proved which we state here as a lemma.

\begin{lemma}
Let $a=s=1$ and $q<0$ in (\ref{o1}).

(a) If $q>-2$ then all solutions of (\ref{o1}) with $r_{0}>0$ are
well-defined, positive and bounded.

(b) If $-\sqrt{5/2}<q<-\sqrt{2}$ then (\ref{o1}) has an asymptotically stable
2-cycle $\{t_{1},t_{2}\}$ where%
\[
t_{1}=\frac{-q-\sqrt{q^{2}-2}}{2},\quad t_{1}=\frac{-q+\sqrt{q^{2}-2}}{2}%
\]

(c) If $q=-\sqrt{3}$ then (\ref{o1}) has a stable solution of period 3%
\[
r_{0}=\frac{2}{\sqrt{3}}\left(  1+\cos\frac{\pi}{9}\right)  ,\quad r_{1}%
=r_{0}-\sqrt{3}+\frac{1}{r_{0}},\quad r_{2}=r_{1}-\sqrt{3}+\frac{1}{r_{1}}%
\]

(d) If $-2<q\leq-\sqrt{3}$ then solutions of (\ref{o1}) with $r_{0}>0$ include
cycles of all possible periods.

(e) For $-2<q<-\sqrt{3}$ orbits of (\ref{o1}) with $r_{0}>0$ are positive,
bounded and chaotic in the sense of Li and Yorke (see \cite{LY}).
%\cite{LY}.

\end{lemma}

Now we have the following.

\begin{proposition}
\label{P}Assume that conditions (\ref{nt}) and (\ref{d0}) hold with $a=1$,
$b^{\prime\prime}=bD_{cb}^{\prime}$, $c+bb^{\prime}/b^{\prime\prime}=q<0.$

(a) If $q>-2$ then all solutions of (\ref{1}) with $x_{0}+by_{0}+c>0$ are
well-defined and bounded.

(b) If $-\sqrt{5/2}<q<-\sqrt{2}$ then (\ref{1}) has an asymptotically stable
2-cycle $\{(x_{1},y_{1}),(x_{2},y_{2})\}$ where $y_{i}$ is given by
(\ref{pyn}) and%
\[
x_{1}=\frac{-q-\sqrt{q^{2}-2}}{2},\quad x_{2}=\frac{-q+\sqrt{q^{2}-2}}{2}%
\]

(c) If $bb^{\prime}=-b^{\prime\prime}(c+\sqrt{3})$ and $x_{0}+by_{0}%
+c=2\left(  1+\cos\pi/9\right)  /\sqrt{3}$ then the points $(x_{i},y_{i})$,
$i=1,2,3$ constitute a stable orbit of period 3 for (\ref{1}) where $y_{i}$ is
given by (\ref{pyn}) and
\[
x_{1}=\frac{2}{\sqrt{3}}\left(  1+\cos\frac{\pi}{9}\right)  ,\quad x_{2}%
=x_{1}-\sqrt{3}+\frac{1}{x_{1}},\quad x_{3}=x_{2}-\sqrt{3}+\frac{1}{x_{2}}%
\]

(d) If $-2<q\leq-\sqrt{3}$ then orbits of (\ref{1}) with $x_{0}+by_{0}+c>0$
include cycles of all possible periods.

(e) For $-2<c+bb^{\prime}/b^{\prime\prime}<-\sqrt{3}$ orbits of (\ref{1}) with
$x_{0}+by_{0}+c>0$ are bounded and exhibit chaotic behavior.
\end{proposition}

\begin{proof}
Let $p$ be the minimal (or prime period) of a solution $\{r_{n}\}$ of
(\ref{o1}) with $r_{0}>0.$ Then the sequence $\{x_{n}\}$ also has minimal
period $p$ and by (\ref{pyn}) $\{y_{n}\}$ has period $p$. It follows that the
orbit $\{(x_{n},y_{n})\}$ has minimal period $p$. Now with $r_{0}=x_{0}%
+by_{0}+c>0$ statements (a)-(e) are true by the above Lemma.
\end{proof}

We point out that the hypotheses of Proposition \ref{P} are sufficient (but
clearly not necessary) for proving that the system is capable of generating
periodic and complex trajectories. We obtain additional information about the
orbits of (\ref{1}) from results concerning (\ref{xo2}) in a future paper.

Proposition \ref{P} applies to several of the 136 system types listed in
(\ref{sty}) that satisfy conditions (\ref{d0}), thus settling the existence of
periodic orbits or occurrence of complex behavior for those special cases. For
example, the following system
\begin{align*}
x_{n+1}  &  =x_{n}+2y_{n}-2\\
y_{n+1}  &  =\frac{0.75x_{n}+1.5y_{n}}{3x_{n}+6y_{n}-6}%
\end{align*}
which is type (40,49) satisfies Part (b) of Proposition \ref{P} ($q=-1.5$) 
and therefore, has an asymptotically stable 2-cycle
$\{(1,0.75),(0.5,1.25)\}$ (a limit cycle). Changing some of the parameters in the
above system yields the following which satisfies Parts (d) and (f) of
Proposition \ref{P} with $q\approx-1.83$
\begin{align*}
x_{n+1}  &  =x_{n}+2y_{n}-2\\
y_{n+1}  &  =\frac{0.25x_{n}+0.5y_{n}+1}{3x_{n}+6y_{n}-6}%
\end{align*}

This system has periodic orbits of all periods (depending on initial points)
and exhibits Li-Yorke type chaos. It is also worth mentioning that the rather
different type (40,37) system below also satisfies Parts (d) and (f) of
Proposition \ref{P}%
\begin{align*}
x_{n+1}  &  =x_{n}+1.5y_{n}-1.8\\
y_{n+1}  &  =\frac{1}{1.5x_{n}+2.25y_{n}-2.7}%
\end{align*}

\end{document}